\documentclass[12pt,reqno]{amsart}
\usepackage{amsmath,amssymb,amsthm,calc,verbatim,enumitem,tikz,url,hyperref,mathrsfs,cite,fullpage,mathtools}
\usepackage{bbm}
\usepackage{setspace}
\usepackage{latexsym}
\usepackage{color}
\usepackage{enumitem,faktor}
\usepackage[right]{lineno}
\usepackage{dsfont}

\usepackage{theoremref}
\newcommand{\cM}{\ensuremath{\mathcal{M}}}

\newcommand{\eps}{\ensuremath{\varepsilon}}

\newcommand{\se}{\ensuremath{\subseteq}}

\newcommand{\ex}{\text{ex}}

\newcommand{\oldqed}{}
\def\endofFact{\hfill\scalebox{.6}{$\Box$}}
 
\newtheorem{theorem}{Theorem}

\newtheorem{lemma}[theorem]{Lemma}

\newtheorem{claim}[theorem]{Claim}

\theoremstyle{definition}
\newtheoremstyle{case}{}{}{\normalfont}{}{\itshape}{\normalfont:}{ }{}
\theoremstyle{case}

\numberwithin{theorem}{section}

\makeatletter
\def\moverlay{\mathpalette\mov@rlay}
\def\mov@rlay#1#2{\leavevmode\vtop{%
		\baselineskip\z@skip \lineskiplimit-\maxdimen
		\ialign{\hfil$\m@th#1##$\hfil\cr#2\crcr}}}
\newcommand{\charfusion}[3][\mathord]{
	#1{\ifx#1\mathop\vphantom{#2}\fi
		\mathpalette\mov@rlay{#2\cr#3}
	}
	\ifx#1\mathop\expandafter\displaylimits\fi}
\makeatother
\newcommand{\cupdot}{\charfusion[\mathbin]{\cup}{\cdot}}

\begin{document}
	
	\title{On multicolor Tur\'an numbers }
	\author{J\'ozsef Balogh \and Anita Liebenau \and Let\'icia Mattos \and Natasha Morrison}

	\thanks{\rule[-.2\baselineskip]{0pt}{\baselineskip}%
		A.~Liebenau was partially supported by the Australian Research Council (DP).  
		J.~Balogh was partially supported by NSF grants DMS-1764123 and RTG DMS-1937241.
		L.~Mattos was supported by the Deutsche Forschungsgemeinschaft (DFG, German Research Foundation) under Germany's Excellence Strategy – The Berlin Mathematics Research Center MATH+ (EXC-2046/1, project ID: 390685689) and AMS-Simons Travel fund.
		N.~Morrison was supported by NSERC Discovery Grant RGPIN-2021-02511.
		}
	
	\address{ \phantom{a} \newline J\'ozsef Balogh \newline
	Department of Mathematics, University of Illinois at Urbana-Champaign, Urbana, Illinois 61801, USA. \texttt{E-mail:jobal@illinois.edu}}

	\address{ \phantom{a} \newline Anita Liebenau\newline
	School of Mathematics and Statistics, UNSW Sydney, NSW 2052, Australia\newline
	E-mail: \texttt{a.liebenau@unsw.edu.au}}

	\address{ \phantom{a} \newline Letícia Mattos \newline
		Institut für Informatik, Universität Heidelberg, Im Neuenheimer Feld 205, D-69120 Heidelberg, Germany \newline
	    E-mail: \texttt{E-mail:mattos@uni-heidelberg.de}}
	
		\address{\phantom{a} \newline Natasha Morrison\newline
		Department of Mathematics and Statistics, University of Victoria, Victoria, B.C., Canada \newline E-mail: \texttt{nmorrison@uvic.ca}}
	
	\begin{abstract}
		We address a problem which is a generalization of Turán-type problems recently introduced by Imolay, Karl, Nagy and Váli.
		Let $F$ be a fixed graph and let $G$ be the union of $k$ edge-disjoint copies of $F$, namely $G = \cupdot_{i=1}^{k} F_i$, where each $F_i$ is isomorphic to a fixed graph $F$ and $E(F_i)\cap E(F_j)=\emptyset$ for all $i \neq j$.
		We call a subgraph $H\se G$ \emph{multicolored} if $H$ and $F_i$ share at most one edge for all $i$.
		Define $\ex_F(H,n)$ to be the maximum value $k$ such that there exists 
		$G = \cupdot_{i=1}^{k} F_i$ on $n$ vertices without a multicolored copy of $H$.
		We show that $\ex_{C_5}(C_3,n) \le n^2/25 + 3n/25+o(n)$ and that all extremal graphs are close to a blow-up of the 5-cycle.
		This bound is tight up to the linear error term.
	\end{abstract}
	
	\maketitle
	
	\section{Introduction}
	
	For a graph $G$, let $\text{ex}(n,G)$ denote the maximum number of edges in a graph on $n$ vertices that does not contain $G$ as a subgraph.
	The classical Turán theorem~\cite{turan} from 1941 states that $\text{ex}(n,K_{t+1}) = e(T_{n,t})$, where $K_t$ denotes the complete graph on $t$ vertices and $T_{n,t}$ denotes the complete $t$-partite graph on $n$ vertices whose part sizes are as equal  as possible.
	Several multicolored generalizations of Turán-type problems have been studied since then.
	Keevash, Mubayi, Sudakov and Versta\"ete~\cite{KMSV} introduced the concept of the rainbow Turán numbers.
	For a non-bipartite graph $H$, they asymptotically determined the maximum
	number of edges in a graph on $n$ vertices that has a proper edge-coloring with no
	rainbow $H$. 
	Another variant was introduced by Conlon and Tyomkyn~\cite{ConlonTyomkyn}.
	For a graph $F$, they obtained bounds on the minimum number of colors in a proper edge-coloring of $K_n$ that does not contain $k$ vertex-disjoint color isomorphic copies of $F$.

	This paper focuses on a related Turán-type problem recently introduced by Imolay, Karl, Nagy and Váli~\cite{IMOLAY2022112976}.
	Let $F$ be a fixed graph and let $G$ be the union of $k$ edge-disjoint copies of $F$.
	That is, $G = \cupdot_{i=1}^{k} F_i$, where each $F_i$ is isomorphic to a fixed graph $F$ and $E(F_i)\cap E(F_j)=\emptyset$ for all $i \neq j$.
	We call a subgraph $H\se G$ \emph{multicolored} if $H$ and $F_i$ share at most one edge for every $i$.
	Define $\ex_F(H,n)$ to be the maximum  $k$ such that there exists 
	$G = \cupdot_{i=1}^{k} F_i$ on $n$ vertices with no multicolored copy of $H$.
	
	Determining $\ex_F(H,n)$ in general can be a very hard problem.
	For example, $\ex_{C_3}(C_3,n)$ is related to the famous $(6,3)$-problem,
	introduced by Ruzsa and Szemerédi~\cite{RSz} in 1978.
	A variant of the $(6,3)$-problem asks for the maximum number of edges in a graph in which each edge belongs to a unique triangle, i.e., asks to determine $\ex_{C_3}(C_3,n)$.
	The exact asymptotics of $\ex_{C_3}(C_3,n)$ are still unknown.

	For pairs of graphs $(F,H)$ for which there is no homomorphism from $H$ to $F$, 
	Imolay, Karl, Nagy and Váli~\cite{IMOLAY2022112976} showed that the extremal number $\ex_F(H,n)$ is at least $n^2/v(F)^2-o(n^2)$.
	The construction is based on packing copies of $F$ into a blow-up of $F$ on $n$ vertices, whose parts have sizes as equal  as possible.
	Imolay, Karl, Nagy and Váli~\cite{IMOLAY2022112976} also 
	proposed the problem of determining the set of pairs $(F,H)$ for which $\text{ex}_F(n,H) = n^2/v(F)^2+o(n^2)$, and suggested that this should be the case for $F=C_5$ and $H=C_3$. 
	Recently, Kovács and Nagy~\cite{kovacs2022multicolor} showed that this is indeed true for $\ex_{C_5}(C_3,n)$. 
	Moreover, from their methods it follows that $\ex_{C_5}(C_3,n) \le n^2/25 + 3n/10$.
	
	It is natural to conjecture that $\ex_{C_5}(C_3,n)$ is precisely equal to the maximum number of edge-disjoint copies of $C_5$ in a balanced blow-up of $C_5$ on $n$ vertices. Therefore, depending on the residue of $n$ modulo 5, we  expect a linear order term to be present in the bounds on $\ex_{C_5}(C_3,n)$.
	For more details, see the discussion in Section~\ref{sec:approximate}.
	
	At the same time that Kovács and Nagy~\cite{kovacs2022multicolor} announced their result, we obtained a similar upper bound for $\ex_{C_5}(C_3,n)$. By combining our methods with theirs, we were able to improve the linear order term.

	\begin{theorem}\label{thm:main}
		For every $\delta > 0$  there exists $n_\delta \in \mathbb{N}$ such that for every $n \ge n_\delta$, we have
		$$\ex_{C_5}(C_3,n) \le \dfrac{n^2}{25}+ \dfrac{3n}{25}+\delta n.$$
	\end{theorem}
	
	Say that a graph $G$ is an extremal multicolored triangle-free graph if  the edge set of $G$ can be partitioned into $\ex_{C_5}(C_3,n)$ edge-disjoint copies of $C_5$ and $G$ does not contain a multicolored copy of $C_3$.
	Our second result captures the structure of all extremal multicolored triangle-free graphs. For $i\in[5]$, we denote by $G(A_i,A_{i+1})$ the bipartite subgraph of $G$ on vertex set $A_i\cup A_{i+1}$ and edges $uv\in E(G)$ where $u\in A_i$ and $v\in A_{i+1}$. 
	
	\begin{theorem}\thlabel{thm:main-structure}
		For every $\delta >0$ there exists $n_\delta \in \mathbb{N}$ such that the following holds for every $n\ge n_\delta$.
		Let $G$ be an extremal multicolored triangle-free graph on $n$ vertices.
		Then there exists a partition $V(G) = A_1 \cupdot \cdots \cupdot A_5$ such that all but $2n/5+\delta n$ edges of $G$ belong to $\bigcup_{i} G(A_i,A_{i+1})$. Moreover, for every $i \in [5]$ we have
		\[n/5 - 2^4 \le |A_i| \le n/5 +2^6.\]
	\end{theorem}
	
	Our bound on the number of edges that do not belong to $\bigcup_{i} E_G(A_i,A_{i+1})$ 
	is best possible. One may create a small perturbation of a blow-up of $C_5$  to generate another extremal graph with $2n/5$ edges across different parts, we provide the details in Section~\ref{sec:approximate}. 

	The paper is organised as follows. In Section 2, we deduce an approximate version of Theorem~\ref{thm:main}, but with a worse linear error term. This approximate estimate is used in Section 3, where we show that all extremal examples are close to a blow-up of $C_5$, up to a small quadratic error term. We also deduce some other properties of extremal graphs.
	In Section 4, we prove Theorem~\ref{thm:main-structure}.
	Finally, in Section 5 we prove Theorem~\ref{thm:main}.

	\section{An approximate asymptotic result}\label{sec:approximate}
	
	In this section, we prove a version of Theorem~\ref{thm:main} with a slightly worse linear error term, see Theorem~\ref{thm:mainlinearapprox}.
	From now on we let $G$ be a graph with vertex set $[n]$ whose edge set is written as a  union of $\text{ex}_{C_5}(C_3,n)$ edge-disjoint copies of $C_5$.
	We denote this decomposition by $G = \cupdot_{j=1}^{k} F_j$, where each $F_j$ is a copy of $C_5$ and $k = \ex_{C_5}(C_3,n)$.
	Moreover, we identify each $F_j$ with the color $j$ and denote the function associated with this coloring by $c:E(G)\to \mathbb{N}$.
	
	\subsection{A discussion on lower bounds}\label{sec:lower-bounds}
	By Theorem 3.1 in~\cite{IMOLAY2022112976}, we have that $\ex_{C_5}(C_3,n) \ge n^2/25-o(n^2)$.
	Here, we need a more precise estimate.
	We say that a graph $H$ is a {\it blow-up} of $C_5$ if there exists a partition of $V(H)=A_1 \cupdot \cdots \cupdot A_5$ 
	such that $$E(H)=\bigcup_{i \in [5]}\{ab: a \in A_i \text{ and } b \in A_{i+1}\},$$
	where $[n]:=\{1,\ldots,n\}$.
	Whenever we are dealing with parts $A_1,\ldots,A_5$, indices are always interpreted modulo $5$.
	For each $a = (a_1,\ldots,a_5) \in \mathbb{N}^5$, define $b(a) = \min_i a_i a_{i+1}$.
	Observe that if we have a blow-up of $C_5$ with parts of sizes $a_1,\ldots,a_5$, then $b(a)$ is the minimum number of edges between two consecutive parts.

	It is not hard to show that if $G$ is a subgraph of the blow-up of $C_5$ with parts of size $(a_1,\ldots,a_5)$, then $k \le b(a)$ and equality can be attained.
	Moreover, under the restriction $a_1+\ldots+a_5=n$, $b(a)$ is maximized by a balanced blow-up, that is when the parts have as equal size as possible.
	Define
	\[t(n):=\max \Big \{b(a): a \in \mathbb{N}^5, a_1 + \ldots + a_5 =n \Big \}.  \] 
	By writing $n = 5q+r$, where $r \in \{0,\ldots,4\}$,
	one can check that $t(n)=q^2$ if $r \in \{0,1,2\}$ and $t(n)=q(q+1)$ if $r \in \{3,4\}$.
	In particular, we have
	\begin{align}\label{eq:lower-bound-extremal}
		\dfrac{n^2}{25} - \dfrac{2n}{5} \le q^2+q \mathds{1}_{r \in \{3,4\}}= t(n) \le \ex_{C_5}(C_3,n).
	\end{align}

We now show that the term $2n/5$ in \thref{thm:main-structure} is tight by providing graphs on $n$ vertices whose edge sets are decomposed into $t(n)$ monochromatic copies of $C_5$ without a multicolored triangle (they do contain $2n/5$ (non-multicolored) triangles). We do so by perturbing a balanced blow-up of $C_5$. 
To describe the construction, say that $x_1x_2x_3x_4x_5$ is a copy $F$ of $C_5$ if $x_1,\ldots, x_5$ are the vertices of $F$, and edges are $x_ix_{i+1}$, for all $i=1,\ldots,5.$
For simplicity, assume that $n$ is a multiple of five, let $q=n/5$, and let $A_1,\ldots,A_5$ be disjoint sets, each of size $q$. 
Let $(v_j^i)_{j \in [q]}$ be a labeling of the vertices of $A_i$, for each $i \in [5]$. 
For each $i,j\in [q]$, let $F_{i,j}=v_i^1v_j^2 v_i^3 v_j^4 v_{i+j}^5$ be a copy of $C_5$.
Clearly, this is a collection of $q^2$ edge-disjoint 5-cycles forming a balanced blow-up of $C_5.$ Now, for every $i \in [q]$, remove the cycle $v_i^1 v_i^2 v_i^3 v_i^4 v_{2i}^5$, and replace it by $v^1_iv^3_iv^2_iv^4_iv^5_{2i}$. That is, we remove the edges $v^1_iv_i^2$ and  $v^3_iv^4_i$, and replace them by  $v^1_iv^3_i$ and $v^2_iv^4_i$. Let us call these new edges {\em crossing edges}.
This switch of edges will not create a multicolored triangle.
Indeed, observe that the crossing edges form a matching of size $2q=2n/5$.
The only triangles that are created are those with vertices in $A_1$, $A_2$ and $A_3$ or with vertices in $A_2$, $A_3$ and $A_4$.
Moreover, the triangles are of the form $v^{1+k}_iv^{2+k}_jv_i^{3+k}$, for some $i,j,k$, which means that the edges $v^{1+k}_iv^{2+k}_j$ and $v^{2+k}_jv_i^{3+k}$ have the same color. Therefore, the constructed graph does not contain a multicolored triangle.

	\subsection{Upper bounds} 
	If a graph $G$ is the union of edge-disjoint copies of $C_5$, each receiving a distinct color, then it is clear that every vertex $v \in V(G)$ has even degree,  and that in each color, $v$ is incident to exactly zero or two edges. 
	If $G$ is an extremal multicolored triangle-free graph then $G$ does not have multicolored triangles. In general, however, $G$ may have many triangles formed by two color classes, as exemplified at the end of the previous subsection. 
	Our next lemma states that for every $v \in V(G)$, the number of edges inside the neighborhood of $v$ is at most $3d(v)/2$,  hence $G$ does not have many triangles.
	
	\begin{lemma}\thlabel{lemma:edgesinsideneighborhood}
		For each $v \in V(G)$, there are at most $3d(v)/2$ edges inside $N_G(v)$. Consequently, there are at most $e(G)\le n^2/2$ triangles in $G$.
	\end{lemma}
	
	\begin{proof}
		Let $v \in V(G)$ and $u_1,\ldots, u_{d(v)}$ be the neighbors of $v$.
		Without loss of generality, suppose that $c(vu_i)=c(vu_{i+d(v)/2})=i$,
		for all $i \in [d(v)/2]$.
		Let $\cM = \{\{u_i,u_{i+d(v)/2}\}: i \in [d(v)/2]\}\}$
		be the collection of pairs of vertices which form a monochromatic cherry with $v$.
		Note that these pairs could be edges in $G$, as an edge of the form $u_i u_{i+d(v)/2}$ does not create a multicolored triangle with $v$.
		Let $xy$ be an edge of $G$ inside $N_G(v)$ which does not belong to $\cM$.
		Then, $c(vx)\neq c(vy)$ and hence we must have $c(xy)=c(vx)$ or $c(xy)=c(vy)$.
		Therefore, every edge $e$ inside $N_G(v)$ which is not in $\cM$, must have an endpoint $z$ such that $c(e) = c(zv)$.
		
		Now, fix some color $j \in [d(v)/2]$ (i.e.~a color incident to $v$). We claim that there are at most two edges in color $j$ inside $N_G(v)$ that are not in $\cM$. Indeed, suppose that there are three such edges. 
		These must form a path $u_jz_{1}z_{2}u_{j+d(v)/2}$ inside $N_G(v)$ contained in the color-$j$ copy of $C_{5}$, for some vertices $z_1,z_2 \in N_G(v)$. Then $c(vz_{1}) = c(vz_{2})$, as otherwise there would be a multicolored triangle in $G$. But this means that $\{z_{1}z_{2}\}\in\cM$. 
		
		As every edge inside $N_G(v)$ not in $\cM$ must have a color in $[d(v)/2]$ and as every such color appears at most two times outside $\cM$, we conclude that the number of edges inside $N_G(v)$ is at most $|\cM|+2 d(v)/2 = 3d(v)/2$.
	\end{proof}

	Let us briefly argue that \thref{lemma:edgesinsideneighborhood} is best possible. 
	Consider a 2-coloring of the edges of $K_5$ in which each color class forms a copy of  $C_5$. In this coloring, every vertex $v$ has $3d(v)/2 = 6$ edges inside its neighborhood.
	This  construction can be extended to larger graphs as follows. Let $G$ be a graph that consists of $n$ copies of $K_5$, all sharing exactly one vertex. 
	Color each of these copies of $K_5$ with two colors each, all distinct, in such a way that each monochromatic component is a copy of $C_5$.
	The degree of the common vertex is $4n$ and the number of edges in its neighborhood is $6n = 3d(v)/2$.

	 For a graph $G$ and $v \in V(G)$, define
	 \begin{align}\label{eq:defn-sv}
	 	s_v := d(v)-\dfrac{2e(G)}{n}.
	 \end{align}
	 The function $s_v$  is the difference between the degree of the vertex $v$ and the average degree of $G$,
	 in particular, $\sum_v s_v = 0$.
	 Our next theorem, which is the main result in this section, is a version of Theorem~\ref{thm:main} with a slightly worse linear error term.
	 The upper bound on $\sum_v s_v^2$ shall be used later to bound the number of vertices of small degree in extremal graphs, which is crucial in the proof of Theorem~\ref{thm:main-structure}.

	\begin{theorem}\thlabel{thm:mainlinearapprox}
		Let $q \in \mathbb{N}$, $r \in \{0,\ldots,4\}$ and set $n = 5q+r$.
		If $n \ge 26$, then
		\begin{align}\label{eq:uppermainlinearapprox}
			\ex_{C_5}(C_3,n)\le q^2 +q \left (6+\dfrac{8r}{5}-3\cdot \mathds{1}_{r \in \{3,4\}}+o(1)\right )-\dfrac{1}{n}\sum \limits_{v \in [n]}s_v^2.
		\end{align}
		Furthermore,
		\begin{align}\label{eq:uppermainlinearapprox-2}
			\sum \limits_{v \in [n]}s_v^2 \le \left(6+\dfrac{8r}{5}+o(1)\right)qn.
		\end{align}
	\end{theorem}
	\begin{proof}
		Let $G$ be an extremal multicolored triangle-free graph on~$n$ vertices.
		We find a large bipartite graph~$B \se G$ and use the fact that $G \setminus B$ still needs to contain an edge from every~$F_i$, as~$C_5$ is not bipartite.
		For each $v \in [n]$, define~$B_v$ to be the bipartite graph induced by the edges between $N(v)$ and $[n]\setminus N(v)$.
		For every $v \in [n]$ we have
		\begin{align*}
			e(B_v) = \sum_{u \in N(v)}d(u)-2e(N(v)) \ge  \sum_{u \in N(v)}d(u)-3d(v),
		\end{align*}
		by~\thref{lemma:edgesinsideneighborhood}. 
		Taking the average over every vertex $v$, we obtain
		\begin{align}\label{eq:avgBv}
			\dfrac{1}{n}\sum \limits_{v\in [n]}e(B_v) \ge \dfrac{1}{n}\sum \limits_{v\in [n]} \sum_{u \in N(v)}d(u)-\dfrac{6e(G)}{n}.
		\end{align}
		Note that for each $u \in [n]$ the degree $d(u)$ appears exactly $d(u)$ times in the double sum.
		Thus,
		\begin{align}\label{eq:avgBv-2}
			\dfrac{1}{n}\sum_{v \in [n]} \sum_{u \in N(v)} d(u) = \dfrac{1}{n} \sum \limits_{v \in [n]} d(v)^2 = \dfrac{1}{n} \sum \limits_{v \in [n]} \left(\dfrac{2e(G)}{n}+s_v\right)^2 =  \dfrac{4e(G)^2}{n^2}+\dfrac{1}{n}\sum \limits_{v \in [n]}s_v^2,
		\end{align}
		where we used  that $\sum_v s_v = 0$.
		From~\eqref{eq:avgBv} and~\eqref{eq:avgBv-2}, it follows that there exists a vertex $v \in [n]$ for which
		\begin{align*}
			e(B_v)\ge  \frac{4e(G)^2}{n^2}+\dfrac{1}{n}\sum \limits_{v \in [n]}s_v^2-\frac{6e(G)}{n}.
		\end{align*}
		That is, there exists a bipartite graph $B$ such that 
		\begin{align}\label{eq:numberofcolors}
			e(G)-e(B)\le e(G)- \frac{4e(G)^2}{n^2}-\dfrac{1}{n}\sum \limits_{v \in [n]}s_v^2+\frac{6e(G)}{n}.
		\end{align}
	
		The function $f: \mathbb{R}\to \mathbb{R}$ given by
		$f(x)=x-\frac{4x^2}{n^2}+\frac{6x}{n}$ is decreasing in the interval 
		$[\frac{n^2}{8}+\frac{3n}{4} ,\binom{n}{2}]$.	
		As $n = 5q+r$, recall from~\eqref{eq:lower-bound-extremal} that $t(n) = q^2 +q\mathds{1}_{r \in \{3,4\}}$.
		As $e(G) \ge  5t(n)\ge\frac{n^2}{8}+\frac{3n}{4}$ when $n \ge 26$,  the right-hand size of~\eqref{eq:numberofcolors} is at most $f(5t(n))-\frac{1}{n}\sum \limits_{v \in [n]}s_v^2$ for $n$ sufficiently large.
		Now, note that
		\begin{align}\label{eq:tnovern}
				\dfrac{5t(n)}{n} = \dfrac{5q^2+5q\mathds{1}_{r \in \{3,4\}}}{5q+r} = q \cdot \dfrac{1+\frac{\mathds{1}_{r \in \{3,4\}}}{q}}{1+\frac{r}{5q}}.
		\end{align}
		As $1 - \frac{r}{5q} \le \Big ( 1 + \frac{r}{5q} \Big )^{-1} \le 1-\frac{r}{10q}$ for every $q$ sufficiently large, it follows from~\eqref{eq:tnovern} that
		\begin{align}\label{eq:tnovern-2}
			q \left ( 1-\frac{r}{5q} +\left ( 1- \dfrac{r}{5q} \right ) \dfrac{\mathds{1}_{r \in \{3,4\}}}{q}\right )\le \dfrac{5t(n)}{n} \le q+o(q).
		\end{align}
		By plugging~\eqref{eq:tnovern-2} into $f(5t(n))$, we obtain
		\begin{align}\label{eq:upperboundf}
			f(5t(n)) &\le 5q^2 + 5q\mathds{1}_{r \in \{3,4\}} -4q^2 \left ( 1- \dfrac{2r}{5q} + \dfrac{2\cdot \mathds{1}_{r \in \{3,4\}}}{q} \right ) +6q+o(q)\nonumber \\
			& \le q^2 +q \left (6+\dfrac{8r}{5}-3\cdot \mathds{1}_{r \in \{3,4\}}\right )+o(q).
		\end{align}
		Therefore, it follows from~\eqref{eq:numberofcolors} and~\eqref{eq:upperboundf} that
		\begin{align}\label{eq:diffGandB}
			e(G)-e(B)&\le q^2 +q \left (6+\dfrac{8r}{5}-3\cdot \mathds{1}_{r \in \{3,4\}}\right )+o(q) -\dfrac{1}{n}\sum \limits_{v \in [n]}s_v^2.
		\end{align}
		As $B$ is bipartite but none of the $F_i$ are bipartite, it follows that $G\setminus B$ contains an edge of every $F_i$, and hence
		\begin{align}\label{eq:boundexbip}
			\ex_{C_5}(C_3,n) \le e(G)-e(B).
		\end{align}
		The upper bound on $\ex_{C_5}(C_3,n) $ follows by combining~\eqref{eq:diffGandB} and~\eqref{eq:boundexbip}. The upper bound on $\sum_v s_v^2$ follows from the fact that $\ex_{C_5}(C_3,n)  \ge t(n) = q^2+q \mathds{1}_{r \in \{3,4\}}$, by~\eqref{eq:lower-bound-extremal}.
	\end{proof}

	\section{The structure of extremal graphs}

	In this section, we show that $G$ must be close to a blow-up of $C_5$.
	Moreover, the vertex partition given by the blow-up structure is approximately   balanced. 
	At the end of the section, we refine this structure when restricting our graph to the subgraph spanned by vertices with degree close to the average.
	
	\subsection{The approximate structure}
	
	To deduce the approximate structure of the extremal graphs, we shall need the triangle removal lemma due to~Ruzsa and Szemerédi~\cite{RSz}.
	
	\begin{lemma}[Triangle Removal Lemma]\thlabel{cor:maketrianglefree}
		For every $\delta > 0$, there exists $n_\delta \in \mathbb{N}$ and $f(\delta)>0$ such that the following holds for every $n> n_\delta$.
		If $H$ has at most $f(\delta)n^3$ triangles, then we can make $H$ triangle-free by deleting at most $\delta n^2$ edges.
	\end{lemma}

	Another tool that we need is the following stability theorem of Erd\H{o}s, Gy\H{o}ri and Simonovits~\cite{erdHos1992many}.

	\begin{theorem}[Erd\H{o}s--Gy\H{o}ri--Simonovits]\thlabel{egs}
		For every $\beta >0$ there exists $\gamma=\gamma(\beta) > 0$ and $n_\beta$ such that if $H$ is an $n$-vertex triangle-free graph with $n>n_\beta$  and $e(H)\ge n^2/5-\gamma n^2$, then $H$ can be turned into a subgraph of a blow-up of $C_5$ by deleting at most $\beta n^2$ edges.
	\end{theorem}

	Now we  deduce our first structural property on extremal multicolored triangle-free graphs, which states that they  are close to being a blow-up of $C_5$.

	\begin{lemma}\thlabel{cor:roughpentagon-like}
			For every $\eps > 0$ there exists $n_\eps \in \mathbb{N}$ such that for every $n > n_\eps$ and an
		 extremal multicolored $n$-vertex triangle-free graph $H$, 
		 $H$ can be turned into a subgraph of a blow-up of $C_5$ by deleting at most $\eps n^2$ edges.
	\end{lemma}
	\begin{proof}
	In the proof we assume that $n$ is sufficiently large, in order that all previous lemmas are applicable. 
		For a given $\eps >0$ set $\beta:=\eps/2$ and let $\gamma=\gamma(\beta) > 0$ be given by~\thref{egs}. Set $\delta:=\min\{\beta,\gamma\}$.
		By~\thref{lemma:edgesinsideneighborhood}, there are at most $n^2/2 = o(n^3)$ triangles in $H$,  hence we can use the Triangle Removal Lemma (Lemma~\ref{cor:maketrianglefree}).
		As $\ex_{C_5}(C_3,n) \ge t(n) \ge n^2/25 -2n/5$ by~\eqref{eq:lower-bound-extremal}, we have $e(H) \ge n^2/5 - 2n$.
		 ~\thref{cor:maketrianglefree} applied with $\delta/2$ implies that
		there exists a triangle-free subgraph $H' \se H$ such that 
		\begin{align*}
			e(H')&\ge n^2/5-2n- \delta n^2/2 \ge 
			n^{2}/5-\gamma n^2,
		\end{align*}
		for $n$ large enough and by the choice of $\delta.$ Thus, $H'$ can be turned into a subgraph of a blow-up of $C_5$ by deleting at most $\beta n^2$ edges,  by~\thref{egs}. 
		Together, these imply that there exists  $H''\se G$ which is a subgraph  of a blow-up of $C_5$  and such that $e(H'')\ge e(H)-(\beta+\delta) n^2 \ge e(H)-\eps n^2$.
	\end{proof}

	From now on, we denote by $A_1,\ldots, A_5$ the disjoint sets corresponding to the partition of $V(G)$ given by a subgraph of a blow-up of $C_5$ as in~\thref{cor:roughpentagon-like}.
	We denote by $C_5(A_1,\ldots,A_5)$ the graph which is a blow-up of  $C_5$, with parts $A_1,\ldots,A_5$.
	Moreover, we assume that the intersection of $G$ with the blow-up $C_5(A_1,\ldots,A_5)$ gives a subgraph of a blow-up of $C_5$ in $G$ with the maximum number of edges. Recall that for $i\in[5]$, we denote by $G(A_i,A_{i+1})$ the bipartite subgraph of $G$ on vertex set $A_i\cup A_{i+1}$ and edges $uv\in E(G)$ where $u\in A_i$ and $v\in A_{i+1}$. 
	
	Our next lemma states that $A_1,\ldots,A_5$ is close to being an equipartition and that each $G(A_i,A_{i+1})$ is close to being a complete bipartite graph.
	
	\begin{lemma}\thlabel{lemma:sizeofAis}
		For every $\eps > 0$ there exists $n_0 \in \mathbb{N}$ such that if $n > n_0$, then the following holds.
		For every $i \in [5]$ we have 
		\begin{align*}
			 e_G(A_i,A_{i+1})\ge n^2/25 - \eps n^2 \qquad
			\text{and}
			\qquad
			n/5 - \eps n \le |A_i| \le  n/5 + \eps n.
		\end{align*}
	\end{lemma}
	
	\begin{proof} For a given $\eps>0$ let $\delta = \eps/2^{11}$ and assume without loss of generality that $n$ is large enough such that $H \se C_5(A_1,\ldots,A_5)$ is a subgraph of $G$ with at least $e(G)-\delta n^2$ edges (the existence of such  $H$ follows from~\thref{cor:roughpentagon-like}).
	This implies that there are at least $e(G)/5-\delta n^2\geq n^2/25-5\delta n^2$ monochromatic copies of $C_5$ in $H$, and hence
		\begin{align}\label{eq:edgesinbetween}
			|A_i||A_{i+1}|\ge e_G(A_i,A_{i+1})\ge n^2/25 - 5\delta n^2
		\end{align}
		 for every $i \in [5]$.

		Now, suppose  that $|A_1| = n/5 - \alpha n $ for some $\alpha \in ( 2^9\delta, 5^{-1})$.
		From~\eqref{eq:edgesinbetween}, we obtain
		\begin{align}\label{eq:sizeofA-1}
		\nonumber	|A_{2}|, |A_{5}| &\ge \dfrac{n^2/25 -5\delta n^2}{n/5 - \alpha n} \ge \dfrac{n^2/25 -5\delta n^2}{n/5}\cdot (1+5\alpha ) =
			 \frac{n}{5} +(\alpha  -25\delta  -125\delta \alpha) n \ge \frac{n}{5}  +\frac{2\alpha n}{3}.
		\end{align}
		Thus, we have
		\[|A_3|+|A_4| = n- |A_1|-|A_2|-|A_5| \le 2n/5 - \alpha n /3.\]
		This implies
		\begin{align*}
			|A_3||A_4| &\le (n/5 -\alpha n /6)^2 = n^2/25 -\alpha n^2/15+\alpha^2 n^2/36 \le n^2/25 -\alpha n^2/30,
		\end{align*} 
		which contradicts~\eqref{eq:edgesinbetween}, as $\alpha/30 > 6 \delta$, proving
	 $|A_i|\ge n/5-2^9\delta n$ for every $i \in [5]$. From this we conclude that $|A_i|\le n-4(n/5-2^9\delta n)\le n/5 +2^{11}\delta n$ for every $i \in [5]$. The lemma follows by choice of $\delta.$
	\end{proof}

	Define $N_{i}(v)$ to be the set of neighbors of $v$ in $G$ which are contained in $A_i$ and let $d_i(v) = |N_{i}(v)|$.
	We refer to the edges not in  $\bigcup_i E_G(A_i,A_{i+1})$ as \emph{unstructured} edges.
	Observe that, as $G \cap C_5(A_1,\ldots,A_5)$ gives a subgraph of a blow-up of $C_5$ in $G$ with the maximum number of edges, the number of unstructured edges is minimum over all subgraphs of a blow-up of $C_5$ in $G$.
	Our ultimate goal is to show that the number of unstructured edges is at most linear.

	Our next lemma gives an upper bound on the degree and the number of unstructured edges incident to each vertex of $G$.

	\begin{lemma}\thlabel{lemma:smallcrossdegree}
		For every $\delta > 0$ there exists $n_\delta \in \mathbb{N}$ such that if $n > n_\delta$, then the following holds.
		For every $i \in [5]$, $v \in A_i$ and $j \notin \{i-1,i+1\}$, we have
		\begin{align}\label{eq:bound-dj-dv}
		d(v)\le 2n/5+\delta n \qquad \text{and}\qquad d_j(v)\le \delta n.
		\end{align}
		Moreover, for $t \in \{i-1,i+1\}$ we have 
		\begin{align}\label{eq:lower-bound-degree}
			d_t(v)\ge d(v)-\dfrac{n}{5}-\delta n.
		\end{align}
	\end{lemma}
	\begin{proof}
		Let $\gamma < \delta/4$ be a small enough constant and set $\eps = \gamma/18$. 
		First, apply Lemma~\ref{lemma:sizeofAis} with parameter $\eps^2/2$ to obtain 
		\begin{align}\label{eq:L3Pt4-applied}
			e_G(A_i,A_{i+1}) &\ge n^2/25 -\eps^2 n^2/2 \text{ and } |A_i|\le n/5+\eps^2 n/2 \text{ for all } i\in[5].
		\end{align}
		Let us now argue that for every $i \in [5]$ and $v \in V(G)$, we have
		\begin{align}\label{eq:usefuleq}
			\min \big \{ d_i(v),d_{i+1}(v)\big\}\le \eps n
		\end{align}
		for $n$  sufficiently  large. 
		Suppose for contradiction that~\eqref{eq:usefuleq} does not hold.
		Since $e(G[N(v)])\le 3n/2$ by Lemma~\ref{lemma:edgesinsideneighborhood}, there are at most $3n/2$ edges $uw$ with $u \in N_i(v)$ and $w \in N_{i+1}(v)$. This implies that the number of edges between $A_i$ and $A_{i+1}$ is at most
		\begin{align*}
			& |A_i||A_{i+1}|-d_i(v)d_{i+1}(v)+3n/2 \le |A_i||A_{i+1}| - \eps^2 n^2+3n/2.
		\end{align*}
		Thus, using~\eqref{eq:L3Pt4-applied} twice, we obtain 
		\begin{align*}
			\frac{n^2}{25} - \frac{\eps^2 n^2}{2} 
			&\le e_G(A_i,A_{i+1}) \le \left (\frac{n}{5}+\frac{\eps^2 n}{2} \right )^2- \eps^2 n^2+\frac{3n}{2}\\
			&= \frac{n^2}{25}+\frac{\eps^2 n^2}{5}+\frac{\eps^4 n^2}{4}-\eps^2 n^2+\frac{3n}{2} 
			 < \frac{n^2}{25}-\frac{\eps^2 n^2}{2}
		\end{align*}
		for $n$  sufficiently  large, a contradiction. 
		Thus, we conclude that ~\eqref{eq:usefuleq} holds.
		
		Now, fix a vertex $v \in V(G)$.
		If $d_G(v)\le \gamma n$, then the condition $d_j(v)\le \delta n$ is trivially satisfied.
		Thus, let us assume that $d_G(v)> \gamma n$.
		From~\eqref{eq:usefuleq} it follows that the set 
		\[S:= \big\{j\in [5]: d_j(v)\le \eps n\big\}\]
		has size at least three, in particular
		 there exists $i\in [5]$ such that $\{i,i+2,i+3\}\se S$.
		Therefore, $v$ can have a large neighborhood only inside the union $A_{i-1}\cup A_{i+1}$. Note that this implies in particular that 
		$$ d(v) \le |A_{i-1}|+|A_{i+1}|+3\eps n \le \frac{2n}{5} + \eps^2 n + 3\eps n \le 2n/5 +\delta n,$$ 
		by~\eqref{eq:L3Pt4-applied}, choice of $\eps$, and for $n$ large enough. That is, the first part of~\eqref{eq:bound-dj-dv} is proved.

		We claim that if $d_{i-1}(v)> 3\eps n$ and $d_{i+1}(v)> 3\eps  n$, then $v\in A_i$.
		Indeed, as $d_i(v)+d_{i+2}(v)+d_{i+3}(v)\le 3\eps n$,
		if we had $v \notin A_i$, then we could move $v$ to $A_i$ and obtain a subgraph of a blow-up of $C_5$ with more edges, a contradiction. In particular, this implies the second part of~\eqref{eq:bound-dj-dv} in this case.  
				 
		Now assume that either $d_{i-1}(v)\le 3\eps n$ or $d_{i+1}(v)\le 3\eps n$.
		As $\{i,i+2,i+3\}\se S$, in both cases we have that there exists an index $k$ for which $d_k(v)\ge d(v)- 6\eps n>12\eps n$, using our assumption that $d_G(v) > \gamma n = 18\eps n$, 
		and $d_j(v)\le 3\eps n$ for every $j \neq k$. 
		As the set $A_k$ contains most of the neighbors of $v$, we must have either $v \in A_{k-1}$ or $v \in A_{k+1}$, otherwise we could move $v$ to one of these sets and obtain a subgraph of a blow-up of $C_5$ with more edges.
		In both cases, we have that if $v \in A_i$ then $d_j(v)\le \delta n$ for every $j \notin \{i-1,i+1\}$, i.e.~\eqref{eq:bound-dj-dv} holds.

		Finally, assuming $v\in A_i$, for some $i\in [5]$, we have for $t \in \{i-1,i+1\}$ that 
		\[d_t(v)\ge d(v)-\max_j |A_j|-3\eps n \ge d(v)-n/5-\delta n,\]
		where we used the fact that $|S|\ge 3$, \eqref{eq:L3Pt4-applied}, and the choice of $\eps$. This completes our proof.
	\end{proof}

\subsection{Cleaning the graph}

	In this section, we obtain a refinement of Lemma~\ref{lemma:smallcrossdegree} for `good' vertices in $G$. Define the set of \emph{good} vertices to be
	\[V_g := \big\{v \in [n]: d_G(v)\ge 7n/20\big\}.\]
	Observe that in a balanced blow-up of $C_5$, each vertex has degree approximately $2n/5$, i.e., every vertex is good.
	Recall that an unstructured edge is an edge of $G$ which is not contained in $\bigcup_i E_G(A_i,A_{i+1})$.
	Let $L$ be the set of unstructured edges with both endpoints are  in $V_g$.
	Our next lemma states that $L$ is a matching. Moreover, $V_g \cap A_i$ is an independent set for every $i \in [5]$.

	\begin{lemma}\thlabel{lemma:noedgesinsideAi}
		For every $\eps > 0$ there exists $n_\eps \in \mathbb{N}$ such that for every $n > n_\eps$
		we have
		 $G[V_g \cap A_i] = \emptyset$ for every $i\in[5]$ and 
			 $L$ is a matching.
			 \end{lemma}

\begin{proof}
		We may assume that $n$ is large enough so that the assertions of \thref{lemma:sizeofAis,lemma:smallcrossdegree} apply. In particular, we assume for all $i\in [5]$, $j\not\in\{i-1,i+1\}$ and $v\in A_i$ that 
		\begin{align}\label{eq:Lem3Pt4andLem3.5}
		|A_i|\le n/5+\eps n,\quad  d_j(v) \le \eps n \quad  \text{and} \quad  d(v)\le 2n/5 +\eps n.
		\end{align}

		Without loss of generality we may assume that $i=1$.
		Let $u, v \in A_1
		 \cap V_g$ and suppose for contradiction that $uv \in E(G)$.
		For $j \in \{2,5\}$, let $A_{j}^{*}$ be the common neighborhood of $u$ and $v$ in $A_{j}$ which avoids the vertices of the $C_5$ of color $c(uv)$.
		Observe that for each $a \in A_{2}^{*}\cup A_{5}^{*}$ we have $c(au)=c(av)$, otherwise $auv$ is  a multicolored triangle.
		Moreover, if $a, a'\in A_{2}^{*}\cup A_{5}^{*}$ are distinct vertices, then 
		$c(au)=c(av)\neq c(a'u) = c(a'v)$,
		otherwise $aua'v$ is   a monochromatic copy of $C_4$, which cannot exist in $G$.
		From this, it follows that there are $|A_{2}^{*}|+|A_{5}^{*}|$ different colors incident to both $u$ and $v$, and hence
		\begin{align}\label{eq:firstboost}
			\min\{d(u),d(v)\}\ge 2\big ( |A_{2}^{*}|+|A_{5}^{*}| \big),
		\end{align}
		where the factor of 2 accounts for the fact that every color contributes twice to a degree of a vertex.
		By the inclusion--exclusion principle (applied separately to each of $A_2^*$ and $A_5^*$), we have
		\begin{align}\label{eq:boundA2A5star}
			|A_{2}^{*}|+|A_{5}^{*}| &\ge d_2(u)+d_5(u)+d_2(v)+d_5(v)-|A_2|-|A_5|-3,
		\end{align}
		where the term $-3$ accounts for neighbors in the copy of $C_5$ of color $c(uv)$. Now, $d_2(v)+d_5(v)\ge d(v) - 3\eps n$, $d_2(u)+d_5(u)\ge d(u) - 3\eps n$, and 
		$|A_2|+|A_5|\le 2n/5+2\eps n,$ all by~\eqref{eq:Lem3Pt4andLem3.5}. 
		Absorbing the constant term and using that $u, v \in V_g$, we thus obtain  
	\begin{align*}
		|A_{2}^{*}|+|A_{5}^{*}| \ge 2\cdot \dfrac{7n}{20}-\dfrac{2n}{5}-9\eps n = \dfrac{3n}{10} -9\eps n.
	\end{align*}
		This, together with~\eqref{eq:firstboost}, implies that $d(u) \ge 3n/5-18\eps n$, which contradicts $d(u) \le 2n/5+\eps n$ in~\eqref{eq:Lem3Pt4andLem3.5}. 
		
		To prove that $L$ is a matching, we split the proof into two cases.
		For the first case, suppose for contradiction that there exist good vertices $v \in A_1$ and $a,b\in A_3$ such that $va, vb \in E(G)$. 
		The common neighborhood of $a,b$ and $v$ inside $A_2$ has size at least
		\(d_2(v)+d_2(a)+d_2(b)-2|A_2|.\)
		Let $B_2^*$ be the common neighborhood of $a,b$ and $v$ inside 
		$A_2$ excluding the vertices of the copies of $C_5$ with colors $c(av)$ and $c(bv)$, so that 
		\[|B_2^*|\ge d_2(v)+d_2(a)+d_2(b)-2|A_2|-6.\]
		Now, $v$ has at least $7n/20-3\eps n - |A_5|$ neighbors in $A_2$, by~\eqref{eq:Lem3Pt4andLem3.5} and since $v$ is a good vertex. Similarly, each of 
		$a$ and $b$ have at least $7n/20-3\eps n - |A_4|$ neighbors in $A_2$. 
		Since $|A_i| \le n/5 + \eps n$ for all $i$, by~\eqref{eq:Lem3Pt4andLem3.5}, we obtain that 
		\[|B_2^*| \ge 3\left(\dfrac{7n}{20}-\dfrac{n}{5}-4\eps n\right)-\dfrac{2n}{5}-2\eps n-6\ge \dfrac{n}{40}.\]	
		In particular, $B_2^*$ is non-empty. Let $u\in B_2^*$. 
		 The edges $uv$, $ua$ and $ub$ have colors different from $c(av)$ and $c(bv)$, by definition of $B_2^*$. This implies that $c(uv)=c(ua)=c(ub)$, as otherwise there was a multicolored triangle in $G$. But this is a contradiction since each color class is a copy of a $C_5$. 

		For the second case, suppose for contradiction that there exist good vertices $v \in A_1$, $a\in A_{4}$ and $b\in A_{3}$ such that $va, vb \in E(G)$.
		Let $D_{5}^*$ be the common neighborhood of $v$ and $a$ in $A_{5}$, which  avoids the vertices of the $C_5$ of color $c(av)$; and let $D_2^{*}$ be the common neighborhood of $v$ and $b$ in $A_{2}$ which avoids the vertices of the $C_5$ with color $c(bv)$. Similarly to the first case, we first show that $D_{5}^*$ and $D_{2}^*$ are non-empty. Again, we have 
	\begin{align}
	|D_{5}^*| &\ge d_5(v)+d_5(a)-|A_5|-6 \nonumber \\
	&\ge \Big(d(v) - |A_{2}| - 3\eps n \Big) +\Big(d(a) - |A_{3}| - 3\eps n \Big) -|A_5|-6 \label{intermediate-bound}\\ 
	&\ge 2\cdot \frac{7n}{20} - \frac{3n}{5} -9\eps n -6\nonumber\\
	&\ge \frac{n}{10} - 10\eps n,\label{eq:D5-bound}
	\end{align}		
	where we used~\eqref{eq:Lem3Pt4andLem3.5} in the second inequality, and 
	the upper bound on $|A_i|$ from~\eqref{eq:Lem3Pt4andLem3.5} and the fact that $a$ and $v$ are good vertices in the third inequality. Similarly, $|D_{2}^*| \ge n/10-10\eps n$.	
	
	Now let $u\in D_{5}^*$ and $u' \in D_{2}^*$ be arbitrary vertices. 
	Then we must have 	\(c(vu)=c(au)\) and  \(c(vu')=c(bu')\)
	since $G$ does not have a multicolored triangle. 
	Moreover, $$|\{c(vu): u \in D_{5}^*\} \cap \{c(vu'): u' \in D_{2}^*\}|\le 1,$$
	as otherwise $G$ contained two monochromatic paths of length four, both with 	endpoints $a$ and $b$. But this cannot happen since each color class is a copy of a $C_5$. 
		It follows that $v$ is incident to at least $|D_{5}^*|+|D_{2}^*|-1$ distinct colors, and hence
		\begin{align}\label{eq:degreeofv}
			d(v)\ge 2\big(|D_{5}^*|+|D_{2}^*|\big)-2\ge 2n/5 -21\eps n, 
		\end{align}
	using~\eqref{eq:D5-bound} and the corresponding bound on $|D_{2}^*|$. 
	Using this new bound on the degree of $v$, it follows from~\eqref{eq:Lem3Pt4andLem3.5} that
		\[d_5(v) \ge d(v) - |A_2| - 3\eps n 
		\ge d(v)-n/5-4\eps n \ge n/5 -25 \eps n.\]
	Feeding this new lower bound on $d_5(v)$ into~\eqref{intermediate-bound}, leaving all other bounds in~\eqref{intermediate-bound}--\eqref{eq:D5-bound} unchanged, we obtain that 
		\begin{align*}
			|D_5^*|&\ge \left(\dfrac{n}{5} - 25 \eps n\right) + \left(\dfrac{7n}{20} -  |A_3| -
			3\eps n\right) - |A_5| -6
			\ge \dfrac{3n}{20} - 30\eps n.
		\end{align*}
	Analogously, one obtains that $|D_2^*| \ge 3n/20 - 30\eps n.$ 
	With~\eqref{eq:degreeofv} this implies now that $d(v)$ is at least $3n/5-121\eps n$,
	which contradicts the upper bound on $d(v)$ in~\eqref{eq:Lem3Pt4andLem3.5}. 
		\end{proof}

	\section{Proof of Theorem~\ref{thm:main-structure}}
	
	Throughout this section, let $\eps\in (0,2^{-200})$ be a fixed small constant and $n$ be sufficiently large (and in particular, large enough such that the conclusions of \thref{lemma:sizeofAis,lemma:smallcrossdegree} hold for $\eps$).
	Recall that we set $A_1,\ldots, A_5$ to be the disjoint sets given by the subgraph of a blow-up of $C_5$ in~\thref{cor:roughpentagon-like} such that 
	the intersection of $G$ with the blow-up $C_5(A_1,\ldots,A_5)$ 
	has the maximum number of edges among all subgraphs of $G$ that are also 
	subgraphs of a blow-up of $C_5$. 
	
	We start by bounding the size of the set of vertices whose degree is far from $2n/5$.
	For $\gamma \in (\eps^{1/4},2^{-12})$, define
	\[V_{\gamma}=\big\{v\in V(G): |d(v)-2n/5| \le \gamma n\big\}.\]
	
	Observe that $V_{\gamma} \se V_g$.
	In particular, Lemma~\ref{lemma:noedgesinsideAi} holds with $V_g$ replaced by $V_{\gamma}$. 
	Our first claim bounds the size of $V_{\gamma}^c:=V(G)\setminus V_{\gamma}$.
	\begin{claim}\thlabel{claim:smallvertices}
		$|V_{\gamma}^c|\le 8/\gamma^2$.
	\end{claim}
	
	\begin{proof}
		Let $n = 5q+r$, where $q \in \mathbb{N}$ and $r \in \{0,\ldots,4\}$. It follows from~\eqref{eq:lower-bound-extremal} and~\eqref{eq:uppermainlinearapprox} that
		\begin{align*}
			\dfrac{n^2}{5}-2n\le e(G)\le \dfrac{n^2}{5} + 12n.
		\end{align*}
		Thus, the average degree of $G$ is $2n/5+O(1)$.
		Recall that for each vertex $v \in V(G)$, we defined $s_v = d(v)-2e(G)/n$.
		If $v \in V_{\gamma}^c$, then
		\begin{align*}
			|s_v| &= |d(v)-2e(G)/n|\ge \gamma n/2.
		\end{align*}
		Therefore, by~\eqref{eq:uppermainlinearapprox-2} (and observing $(6 + 4r/5)q \le 2n$) 
		we have
		\[|V_{\gamma}^c|\cdot \dfrac{\gamma^2 n^2}{4} \le 2n^2,\]
		which implies \[|V_{\gamma}^c|\le \dfrac{8}{\gamma^2}. \qedhere\]
	\end{proof}

	We say that a  $5$-cycle $C$ in $G$ is \emph{great} if all of its edges have the same color and all of its vertices are in $V_{\gamma}$. 
	For each pair of vertices~$a$ and~$b$, let $g_{ab}$ be the number of great 5-cycles containing~$a$ and~$b$.
	\begin{lemma}\thlabel{boundingtab}
		Let $i \in [5]$ and $ab \in E(G)$, with $a \in A_i \cap V_{\gamma}$ and $b \in A_{i+2} \cap V_{\gamma}$.
		 Then, we have $g_{ab}\ge n/5-4\gamma n$.
	\end{lemma}

	\begin{proof}
		Without loss of generality, suppose that $i =1$ and let $a \in A_1 \cap V_{\gamma}$ and $b \in A_3 \cap V_{\gamma}$.
		By Lemma~\ref{lemma:smallcrossdegree}, we have
		\begin{align}\label{eq:boundond2ad2b}
			\min \{ d_2(a), d_2(b)\} &\ge \dfrac{2n}{5}-\gamma n - \dfrac{n}{5}-4\eps n = \dfrac{n}{5}-\gamma n-4\eps n.
		\end{align}
		Thus, the vertices $a$ and $b$ are incident to most  vertices in $A_2$.
		
		Recall that $c:E(G)\to \mathbb{N}$ is the coloring associated to the partition of $E(G)$ into copies of $C_{5}$. 
		Define
		$R_a = \{c(av): av \in E(G), \, v \in A_1 \cup A_3 \cup A_4\}$ 
		and 
		$R_b = \{c(bv): bv \in E(G), \, v \in A_1 \cup A_3 \cup A_5\}$.
		These sets correspond to the set of colors incident to $a$ and $b$, respectively, 
		which appear at an edge which is not contained in 
		$G \cap C_5(A_1,\ldots,A_5)$. 
		By~\thref{lemma:smallcrossdegree}, both $R_a$ and $R_b$ have size at most  $3\eps n$.
		In particular, $|R_a \cup R_b|\le 6\eps n$.
		Let $A_2^*$ be the common neighborhood of $a$ and $b$ inside $A_2$ excluding the vertices incident to some color in $R_a \cup R_b$. 
		Then 
		\begin{align}\label{eq:boundingA2*}
			|A_2^*| &\ge d_2(a)+d_2(b)-|A_2|-5|R_a \cup R_b| \ge  \dfrac{n}{5}-2\gamma n-9\eps n -30\eps n \ge \dfrac{n}{5} - 3\gamma n,
		\end{align}
		 where we use~\eqref{eq:boundond2ad2b}, $|A_2| \le n/5+\eps n$ (c.f.~\thref{lemma:sizeofAis}) and $|R_a \cup R_b|\le 6\eps n$ in the second inequality, and $\gamma>39\eps$ by choice of $\gamma$ in the last inequality. 
		
		As $c(ab)\in R_a \cup R_b$, for every $v \in A_2^{*}$ we have $c(av)\neq c(ab)$ and $c(bv)\neq c(ab)$.
		This implies that $c(av)=c(bv)$ for every $v \in A_2^{*}$.
		In particular, the number of monochromatic 5-cycles containing $a$ and $b$ is at least $n/5-3\gamma n$.
		%
		At most $8/\gamma^2$ of these cycles contain a vertex in $V_{\gamma}^c,$ by \thref{claim:smallvertices}.    
		Therefore, it follows that there are at least $n/5 - 3\gamma n - 8/\gamma^2$ great cycles containing~$a$ and~$b$. 
	\end{proof}

	Recall that a vertex $v$ is good if $d(v)\ge 7n/20$. Recall that an edge $e \in E(G)$ is unstructured if $e \notin \bigcup_i E_G(A_i,A_{i+1})$.
	Let $M$ be the set of unstructured edges.
	By~\thref{lemma:noedgesinsideAi}, we know that 
	$M$ is a matching when restricted to good vertices, and in particular when restricted to~$V_\gamma$.
	With this in mind, one would hope to prove that $|M| \le n/2+o(n)$. 
	It turns out that we can prove a much better bound, which is even close to optimal (as discussed in Section~\ref{sec:approximate}).
	Recall that $\eps\in (0,2^{-200})$ and $\gamma  \in (\eps^{1/4},2^{-12})$.

	
	\begin{lemma}\thlabel{lemma:boundonM}
		Let $q \in \mathbb{N}$ and $r \in \{0,\ldots,4\}$ be such that $n = 5q+r$. Then, we have
		\[|M|\le 2q+2^{6}\gamma q.\]
	\end{lemma}

	\begin{proof}
			First note that the number of unstructured edges with at least one of its endpoints in $V_{\gamma}^c$ 
			is at most 
			\begin{align}\label{eq:naught-type1}
				& 8/\gamma^2 \cdot 3\eps n \le 24 \gamma^2 n \le \gamma q,
			\end{align}
			by \thref{lemma:smallcrossdegree,claim:smallvertices}.
			It remains to bound the number of unstructured edges with both endpoints in $V_\gamma$. 
			
			Let $M_{\gamma}\se M$ be the set of such unstructured edges $ab$ with $a,b \in V_\gamma$, and let $P$ be the set of ordered pairs $(e,C)$ such that $C$ is a great 5-cycle and $e$ is an unstructured edge with both endpoints in~$V(C)$.
			Observe that if $(e,C) \in P$, then $e \in M_{\gamma}$. In particular, 
			\(\sum_{ab \in M_{\gamma}} g_{ab}\le |P|\), where we recall that $g_{ab}$ denotes the number of great 5-cycles containing $a$ and $b$.   For $ab\in M_\gamma$ we must have $a\in A_i$ and $b\in A_{i+2}$  for some $i\in [5]$ (or vice versa), by definition of an unstructured edge and since $G[A_i\cap V_\gamma]$ is empty, by \thref{lemma:noedgesinsideAi}.  Using \thref{boundingtab}, we thus obtain 
			\begin{align}\label{eq:P-lower-bound}
				|P| \ge |M_{\gamma}|\cdot \left(\dfrac{n}{5}-4\gamma n\right).  
			\end{align}
			
			Now, let $s_C$ be the number of unstructured edges with both endpoints in $V(C)$, for each great 5-cycle $C$. 
			The set of unstructured edges spanned by $V_{\gamma}$ is a matching, by~\thref{lemma:noedgesinsideAi},  so we must have $s_C \le 2$ for all $C$. It follows that 
			\begin{align*}
				|P| = \sum_{C \text{ great}} s_C \le \dfrac{2e(G)}{5}.
			\end{align*}
			%
			Combining this bound with~\eqref{eq:P-lower-bound}, we thus obtain  
			\begin{align}\label{eq:double-count-M}
				|M_{\gamma}|\cdot \left(\dfrac{n}{5}-4\gamma n\right)
			\le \dfrac{2\, e(G)}{5} \le 2\, \ex_{C_5}(C_3,n) \le 2q^2 +2q \left (6+\dfrac{8r}{5}-3 \mathds{1}_{r \in \{3,4\}}+o(1)\right )
			\end{align}
	by~\thref{thm:mainlinearapprox}, where we recall that $q$ is defined by $n=5q+r$ and $r\in\{0,\ldots,4\}$. 
		Now, 
		\begin{align}
			\dfrac{1}{n/5-4\gamma n} \le \dfrac{1}{q-4\gamma n} 
			\le \dfrac{1}{q} \left( 1 + 24 \gamma + 2\cdot (24)^2\gamma^2\right)  \le \dfrac{1}{q} ( 1 + 25 \gamma ),
		\end{align}
		where we use $n \le 6q$, $(1-x)^{-1}\le 1+x+2x^2$ for every $x \in (0,1/2)$, and that $\gamma$ is small enough. 
		Combining this last estimate with~\eqref{eq:double-count-M}, we obtain 
		\[|M_{\gamma}| \le 2q+50\gamma q +O(1).\]
		This, together with the bound in~\eqref{eq:naught-type1}, proves our lemma. 
	\end{proof}

From the constructions given in Section~\ref{sec:lower-bounds}, note that we can have $2n/5$ unstructured edges, and hence our bound of $2n/5+o(n)$ on the number of unstructured edges given by Lemma~\ref{lemma:boundonM} is tight.
		Our next lemma improves the bounds on the size of each $A_i$ from an additive linear error term as in \thref{lemma:sizeofAis} to an additive error of constant size. 

	\begin{lemma}\label{lemma:equidistributed}
		Let $q \in \mathbb{N}$ and $r \in \{0,\ldots,4\}$ be such that $n = 5q+r$.
		For every $j \in [5]$, we have 
		\[ q-15 < |A_j| \le q+64.\]
	\end{lemma}

	\begin{proof}
		The proof is similar to the proof of Lemma~\ref{lemma:sizeofAis}.
		Without loss of generality, suppose that $A_1$ is the part of smallest size.
		Suppose for contradiction that $|A_1| = q - i $, for some $i \ge 15$. 
		
		Let $C_1,\ldots, C_k$ be the colored 5-cycles given by the edge-partition of $G$. Note that $k\ge q^2$ since we assume that $G$ is extremal, by~\eqref{eq:lower-bound-extremal}. 
		At most $8/\gamma^2$ of these cycles contain a vertex in $V_\gamma^c$, by \thref{claim:smallvertices}, and at most $2q +2^6\gamma q$ contain an unstructured edge by \thref{lemma:boundonM}. 
		Observe that a $C_j$ with all vertices in $V_\gamma$ that does not contain an unstructured edge is a great cycle $v_1v_2v_3v_4v_5$ such that $v_i \in A_i$ for each $i \in [5]$. Therefore for each $i\in [5]$,
		\begin{equation}\label{eq:edges}
			e(A_i,A_{i+1}) \ge  q^2-2q - 2^6\gamma q - 8/\gamma^2.
		\end{equation} 
		Similarly to the proof of \thref{lemma:sizeofAis},  this implies that
		\begin{align}\label{eq:sizeofA}
			|A_{2}|, |A_{5}| &\ge \dfrac{q^2-2q - 2^6\gamma q - 8/\gamma^2}{q - i} \ge \dfrac{q^2-3q}{q - i} \ge q + i - 3
		\end{align}
		since $\gamma$ is small and $n$ is large. 
		Thus, we have
		\[|A_3|+|A_4| = n- |A_1|-|A_2|-|A_5| \le  2q - i + r +6 \le 2q - i +10.\]
		
		Using this for the upper bound and \eqref{eq:edges} for the lower bound, we obtain
		\begin{align*}
			(q -i/2 + 5)^2 \ge |A_3||A_4| \ge q^2-2q - 2^6\gamma q - 8/\gamma^2, 
		\end{align*} 
		which is a contradiction, as $i \ge 15$.
		As every set $A_1,\ldots, A_5$ has size at least $q-15$, we conclude that $|A_j|\le q + 64$ for any $j \in [5]$.
		This proves the lemma.
	\end{proof}

	Finally, note that Lemmas~\ref{lemma:equidistributed} and~\ref{lemma:boundonM} imply Theorem~\ref{thm:main-structure}.
	
%
	
%
%

\section{Proof of Theorem~\ref{thm:main}}

Let $\eps\in (0,2^{-200})$, $\gamma  \in (\eps^{1/4},2^{-12})$ and set $\delta = 2^8 \gamma$.
We start by bounding the number of triangles in $G$.

\begin{lemma}\thlabel{numberoftriangles}
	The number of triangles in $G$ is at most $2 n^2/25 + 2\delta n^2$.
\end{lemma}

\begin{proof}
	Let $\Delta_{\gamma}$ be the number of triangles where all three vertices belong to $V_{\gamma}$.
	By~\thref{lemma:noedgesinsideAi}, if $a,b,c \in V_{\gamma}$ and they form a triangle, then we must have $a \in A_i$, $b \in A_{i+1}$ and $c \in A_{i+2}$, for some $i \in [5]$.
	Let $M$ be the set of unstructured edges between pairs of good vertices.
	By~\thref{lemma:boundonM} we have $|M|\le 2n/5+\delta n$ and by~\thref{lemma:sizeofAis} we have $\max_i |A_i|\le n/5+\delta n$. As every triangle contains an unstructured edge,
	\[\Delta_{\gamma} \le |M| \cdot \max_i |A_i| \le 2n^2/25+\delta n^2. \]
	
	Let $\Delta_{\gamma}^c$ be the number of triangles containing at least one vertex outside $V_{\gamma}$.
	By~\thref{lemma:edgesinsideneighborhood} and~\thref{claim:smallvertices}, we have 
	\[\Delta_{\gamma}^c \le \sum_{v \in V_{\gamma}^c} \dfrac{3d(v)}{2}\le \dfrac{3n}{2} |V_{\gamma}^c| \le \dfrac{12n}{\gamma^2} \le \delta n^2. \qedhere \]
\end{proof}

Below, we denote by $C_1,C_2,\ldots,C_k$ the set of monochromatic 5-cycles in $G$. 
The next lemma is due to Kovács and Nagy~\cite{kovacs2022multicolor}.
For completeness, we provide their proof here.

\begin{lemma}[Kovács--Nagy]\thlabel{lemma:sumofdeginC}
Let $i\in [k]$ and denote by $\Delta_i^1$ and $\Delta_i^2$  the number of triangles with exactly one  and two edges colored $i$, respectively.
Then,
\[\sum \limits_{v \in C_i} d(v) \le 2n + 2\Delta_i^2 + \Delta_i^1.\]
\end{lemma}

\begin{proof}
Note that $e(G[V(C_i)])=5+\Delta_i^2$. Observe that each vertex in $[n]\setminus C_i$ sends at most two colors to $C_i$, otherwise a multicolored triangle is created. For the same reason, each vertex in $[n]\setminus C_i$ sends at most three edges to $C_i$, and if exactly three are sent then two of them with the same color that must go to adjacent vertices in $C_i$.  
	Thus, we have
	\[\sum \limits_{v \in C_i} d(v) \le 2e(G[V(C_i)])+\Delta_i^1+2(n-5)  \le 2n+2\Delta_i^2 + \Delta_i^1.\qedhere \]
\end{proof}

\begin{proof}[Proof of Theorem~\ref{thm:main}]
	Following Kovács and Nagy~\cite{kovacs2022multicolor}, we estimate the double sum
	\[S:= \sum_{i=1}^{k} \sum \limits_{v \in C_i} d(v).\]
	Every vertex $v$ is contained in $d(v)/2$ monochromatic cycles,  hence  $d(v)$ is  counted $d(v)/2$ times in the double sum above. Therefore, 
	\begin{align}\label{eq:lowerboundS}
		S = \sum \limits_{v \in [n]} \dfrac{d(v)^2}{2} \ge \dfrac{n}{2} \cdot \left(\dfrac{2e}{n}\right)^2 = \dfrac{2e^2}{n}.
	\end{align}
	On the other hand, by~\thref{lemma:sumofdeginC}, we have
	\begin{align}\label{eq:upperboundS}
		S & \le 2nk + \sum_{i=1}^k\big ( 2\Delta_i^2 + \Delta_i^1 \big )  \le 2nk + \dfrac{6n^2}{25}+6\delta n^2,
	\end{align}
	using that each triangle is  counted three times and~\thref{numberoftriangles}.
	As $e(G)=5k$, from~\eqref{eq:lowerboundS} and~\eqref{eq:upperboundS}, it follows that
	\begin{align*}
		\dfrac{50k^2}{n} \le 2kn + \dfrac{6n^2}{25}+6\delta n^2,
	\end{align*}
 hence
	\begin{align*}
		k \le \dfrac{n^2}{25} + \dfrac{3n^3}{625k}+\dfrac{\delta n^3}{k}.
	\end{align*}
	As $k \ge n^2/25-2n$, it follows that $k \le n^2/25 + (3/25+50\delta)n$. As $\epsilon$ (and hence $\delta$) can be chosen to be arbitrarily small, the theorem follows. 
\end{proof}


%

%

\section*{Acknowledgments}

L. Mattos would like to thank Imolay for presenting the problem to her in the Berlin Combinatorics Research Seminar in Summer 2022.
The authors thank the Oberwolfach Workshop Combinatorics, Probability and Computing 2022 for hosting them.
We note that the bound $\ex_{C_5}(C_3,n) = n^2/25+\Theta(n)$ was independently obtained by us (c.f. Theorem~\ref{thm:mainlinearapprox}) before learning about the analogous result of Kovács and Nagy  (Theorem 1.6 in~\cite{kovacs2022multicolor}).
We thank for both referees for careful reading and useful comments.

%

\end{document}